\documentclass[a4paper,draft]{amsproc}
\usepackage{amssymb}
\usepackage[hyphens]{url} 

\theoremstyle{plain}

\theoremstyle{definition}

\theoremstyle{remark}
 
 \numberwithin{equation}{section}
\title{Stability Analysis for the Helmholtz Equation with Many Frequencies}
\author{Elham Sohrabi\\
Division of Mathematics and Computer Science
\\ University of South Carolina Upstate, Spartanburg\\
SC 29303, USA \\
E-mail address: esohrabi@uscupstate.edu}
\date{}

\begin{document}
\maketitle
\newtheorem{theorem}{Theorem}[section]
\newtheorem{lemma}[theorem]{Lemma}
\newtheorem{corollary}[theorem]{Corollary}
\newtheorem{definition}[theorem]{Definition}
\newtheorem{proposition}[theorem]{Proposition}

{\bf Abstract.}

This paper concerns the inverse source problem for the time- harmonic wave equation in a one dimensional domain. The goal is to determine the source function from the boundary measurements. The problem is challenging due to complexity of the Green's function for the Helmholtz equation. Our main result is a logarithmic estimate consists of two parts: the data discrepancy and the high frequency tail.   

\vspace{10pt}


\vspace{10pt}


\section{Introduction and formulation of the problem }
We formulate our problem as following; 
\begin{equation}
\label{PDE}
u(x,\alpha)'' +k^2 u(x,\alpha)=f(x), \quad x\in(-1,1),
\end{equation}
where the direct solution $u$ is required to satisfy the outgoing wave conditions:
\begin{equation}
\label{outgoing cond}
   u'(1,\alpha)+iku(1,\alpha)=0, \qquad   u'(-1,\alpha)-iku(-1,\alpha)=0
\end{equation}

Let $f \in L^2 (-1,1)$, it is well-known that the problem \eqref{PDE}-\eqref{outgoing cond} has a unique solution:   
\begin{equation}
\label{u}
    u(x,\alpha)=\int_{-1}^1 G(x-y) f (y)dy,
\end{equation}
where $G(x-y)$ is the Green function given as follows
\begin{equation}
\label{GreenF}
    G(x-y)=\frac{i}{2}\frac{e^{ik|x-y|}}{k}.
\end{equation}
This article concerns the inverse source problem when the source function $f$ is compactly supported in the interval $(-1,1)$. This paper aims to recover the radiated source $f$ using the boundary measurement $u(1,\alpha )$ and $u(-1, \alpha)$ with $\alpha \in (0,k) $ where $K>1$ is a positive real constant. \\
Motivated by these significant applications, inverse source problems have attracted many researchers from whole part of the world in many areas of science. It has vast applications in acoustical and biomedical/medical sciences, antenna synthesis, geophysics, Elastography and material science (\cite{ABF, B}). It has been known that the data of the inverse source problems for the time harmonic wave equation with a single frequency can not guarantee the uniqueness (\cite{I}, Ch.4). Although, Many studies showed that the uniqueness can be regained by taking many frequency boundary measurement data in a non-empty interval $(0,K)$ noticing the analyticity of wave-field on the frequency \cite{I, J}. Due to the significant applications, these problems have been attracted considerable attention. These kinds of problems have been extensively investigated by many researchers such as In the paper \cite{AI, BLT1,BLRX,  CIL, E, EI, EG, IK,IL,IL1, J2, YL} and \cite{LY}.  It is worth mentioning that this type of problems has also application in important system such classical elasticity system and Maxwell system. For an example, in \cite{EI2, BLZ}, inverse source problems was considered for classical elasticity system with constant coefficients.\\
In this paper, we assume that the domain is homogeneous in the whole space. In this work, we try to establish an estimate to
recover the source functions for the inverse source problem for the one-dimensional Helmholtz equation. In this work, paper source function $f \in H^2((-1,1))$ is assumed to has a support in the domain or  $suppf \subset (-1,1) $. Our main result is as follows;

\begin{theorem}
\label{maintheorem}
Let's $u\in H^2 ((-1,1))$ be the solution of the \eqref{PDE}, Then there exists a generic constant $C$ depending on the domain $(-1,1) $ such that 

\begin{equation}
\label{Istability}
\parallel f\parallel_{(0)}^{2}(-1,1) \leq   C\Big(\epsilon ^2+\frac{M^{2}}{1+K^{\frac{2}{3}}E^{\frac{1}{4}}} \Big), 
\end{equation}
where $K>1$. Here 

\begin{equation*}
\label{epsilon}
\epsilon ^2  = \int _{0} ^{K} \alpha ^2\big(  | u(1,\alpha) |^{2} + | u(-1,\alpha) |^{2} \big ) d\alpha,
\end{equation*}
$E=-ln\epsilon $ and $M = \max \big \{ \parallel f \parallel_{(1)}^2(-1,1) , 1 \big\}$ where $\parallel . \parallel _{(l)}(\Omega)$ is the standard Sobolev norm in $H^l(\alpha)$.
\end{theorem}
\begin{center}
\begin{equation*}
f_1 = \begin{cases}
            	f & \mbox{if } x > 0, \\
0 & \mbox{if } x < 0,
       \end{cases} \quad
f_{2} = \begin{cases}
            0  & \text{if } x > 0, \\
            f & \text{if } x < 0.
       \end{cases}
\end{equation*}

 \end{center}
 {\bf Remark 1.1: } Our main estimate in \eqref{Istability} consists of two parts: the data discrepancy and the high frequency part. The first part is  the Lipschitz part comes form the boundary measurement data  type. The second part is of logarithmic type. Our results shows when the wave number K grows, the problem is more stable.  The estimate \eqref{Istability} also implies the uniqueness of the inverse source problem as the norm of data $\epsilon \rightarrow0$.

\section{Increasing Stability for the Inverse source Problem}
Let 's define 

\begin{equation*}
    I(k)=I_1(k)+I_2(k)
\end{equation*}
where 
\begin{equation}
\label{I1-I2}
I_1(k)= \int _{0} ^{k} \alpha ^2 | u(-1,\alpha) |^{2} d\alpha, \qquad  I_2 (k)= \int _{0}^{k}  \alpha ^2 | u(1,\alpha) |^{2} d\alpha,
\end{equation}
using \eqref{u} and a simple calculation shows that 
\begin{equation}
\label{alphau1u-1}
    \alpha u(1,\alpha)=\int_{0}^{1}\frac{i}{2}e^{i\alpha(1- y)}  f_1(y)dy,  \qquad   \alpha u(-1,\alpha)=\int_{-1}^{0}\frac{i}{2}e^{i\alpha (-1- y)}  f_2(y)dy, 
\end{equation}
where $y \in (-1,1)$. Functions $I_1$ and $I_2$ are both analytic with respect to the wave number $k \in \mathbb{C} $ and play important roles in relating the inverse source problems of the Helmholtz equation and the Cauchy problems for the wave equations. \\
\begin{lemma}
Let $supp f \in (-1,1)$ and $f \in H^1 (-1,1) $. Then 
\begin{equation}
\label{boundI1}
|I_1 (k)|\leq C\Big( |k| \parallel f \parallel_{(0)}^2 (-1,1)  \Big)e^{2|k_2|},
\end{equation}
\begin{equation}
\label{boundI2}
|I_2(k)|\leq C\Big( |k| \parallel f \parallel_{(0)}^2 (-1,1)  \Big)e^{2|k_2|}.
\end{equation}

\end{lemma}
\begin{proof}   

Since we have $k= k_1 +k_2 i$ is complex analytic on the set $\mathbb{S}\setminus [0,k]$, where $\mathbb{S}$ is the sector $S=\{ k \in \mathbb{C}:|$arg$\, k| < \frac{\pi}{4}    \}$ with $k=k_1 +ik_2$. Since the integrands in \eqref{I1-I2} are analytic functions of $k$ in $ \mathbb{S}$, their integrals with respect to $\alpha$ can be taken over any path in $\mathbb{S}$ joining points $0$ and $k$ in the complex plane. Using the change of variable $\alpha=ks$, $s\in (0,1)$ in the line integral \eqref{u}, the fact that $y\in(-1,1)$.
\begin{equation}
\label{I1}
I_1 (k)= \int _{0} ^{1} ks \big | \int_{0}^{1}\frac{1}{2}e^{i(ks)(1- y)}  f_1(y)dy   \big |^2ds,
\end{equation}
and 
\begin{equation}
\label{I2}
    I_2 (k)= \int _{0} ^{1}ks \big | \int_{-1}^{0}\frac{1}{2}e^{i(ks)(-1- y)}  f_2(y)dy   \big |^2ds.
\end{equation}
Noting
\begin{equation*}
    |e^{ iks (-1- y)}|\leq e^{2|k_2|}, \quad |e^{ iks (1- y)}|\leq e^{2|k_2|},
\end{equation*}                                                                                                                                                                                                                                                     
by the Cauchy-Schwartz inequality and  integrating with respect to $s$ and using the bound for $|k|$ in $\mathbb{S}$, the proof of \eqref{boundI1} is complete. Using the similar technique, the  prove the \eqref{boundI2} is straight forward. 

\end{proof}
Note that the functions $I_1(k), I_2(k)$ are analytic functions of wave number$k=k_1+ik_2 \in \mathbb{S}$ and $|k_2|\leq k_1$. The following steps are important to connect the unknown $I_1(k)$ and $I_2(k)$ for $k\in [K,\infty)$ to the known value $\epsilon$ in \eqref{PDE}.\\ Clearly

\begin{equation}
\label{Ie^}
|I_1(k) e^{-2 k}|\leq C\Big( |k_1| \parallel f \parallel_{(0)}^2 (-1,1)\Big)e^{-2 k_1} \leq C M^{2},
\end{equation}
where $M=\max \big\{\parallel f\parallel_{(0)}^2  (-1,1) , 1 \big \}$. With the similar argument bound \eqref{Ie^} is true for $I_2 (k)$.\\
Noting that 
\begin{equation*}
|I_1(k)  e^{-2 k}|\leq \epsilon ^2, \quad |I_2(k)  e^{-2 k}|\leq \epsilon ^2  \textit{ on } [0, K].
\end{equation*}
Defining $\mu (k) $ be the harmonic measure of the interval $[0,K]$ in $\mathbb{S}\backslash [0,K], $ then as known (for example see \cite{I}, p.67), from two previous estimates and analyticity of the functions $I_{1}(k) e^{-2 k}$ and $ I_{2}(k) e^{-2 k} $, we have that 
\begin{equation}
\label{I1Epsilon}
|I_1(k) e^{-2 k}|\leq  C\epsilon ^{2 \mu(k)} M^{2},
\end{equation}
when $K<k< +\infty$. Similarly it is easy to see that 
\begin{equation}
\label{I2Epsilon}
|I_2(k) e^{-2k}|\leq  C\epsilon ^{2 \mu(k)} M^{2},
\end{equation}
consequently
\begin{equation}
\label{I2Epsilon}
|I(k) e^{-2 k}|\leq  C\epsilon ^{2 \mu(k)} M^{2}.
\end{equation}

To obtain the lower bound for the harmonic measure $\mu (k)$, we use the following technical lemma. The proof is  in \cite{CIL}.
\begin{lemma}
Let $\mu (k) $ be the harmonic measure function of the interval $[0,K]$ in $\mathbb{S}\backslash [0,K]$, then 
\begin{equation}
 \begin{cases}
            \frac{1}{2}\leq \mu (k), & \mbox{if } \quad  0<k<2^{\frac{1}{4}}K, \\
\frac{1}{\pi} \Big( \big ( \frac{k}{K} \big)^{4} -1 \Big)^{\frac{-1}{2}} \leq \mu(k), & \mbox{if } \quad 
 2^ {\frac{1}{4}}  K <k .
       \end{cases} 
\end{equation}
\end{lemma}
\vspace{0.5cm}

\begin{lemma}
Let external source function $f\in L^2 (-1,1)$ where $supp f \subset (-1,1)$, then 
\begin{equation*}
\parallel f\parallel^{2} _{(0)} (-1,1) \leq C \int_{0}^{\infty} \alpha ^2 \big ( |u(-1,\alpha)|^2 + |u(1,\alpha)|^2 \big)d\alpha.
\end{equation*}
\end{lemma}
\begin{proof}
Using the result of the paper  \cite{YL} by applying the Green function \eqref{GreenF} and letting $k_1=k_2=k$.
\end{proof}

\begin{lemma}
Let source function $f\in L^2 (-1,1)$, then 
\begin{equation*}
    \alpha^2 |u(-1,\alpha)|^2  \leq C \Big | \int _{-1}^{0}e^{2\alpha y } f_2 (y)dy\Big |^ 2
\end{equation*}
\begin{equation*}
     \alpha^2  |u(1,\alpha)|^2 \leq C \Big | \int _{0}^{1}e^{2\alpha y } f_1 (y)dy\Big |^ 2
\end{equation*}
\end{lemma}
\begin{proof}
It follows from \eqref{alphau1u-1} and $y\in (-1,1)$.
\end{proof}

\section{Increasing stability for inverse source problem for the higher frequency }
To proceed the estimate for reminders in \eqref{I1} and \eqref{I2} for $(k, \infty)$, we need the following lemma.
\begin{lemma}
\label{BoundK-Infty}
Let $u$ be a solution to the forward problem \eqref{PDE} with $f_1 \in H^1(\alpha )$  with $supp f  \subset (-1,1)$, then 
\begin{equation}
\label{lemma4.1}
\int_{k}^{\infty} \alpha ^2 | u(-1,\alpha) |^2 d\alpha+ \int_{k}^{\infty}   \alpha ^2|u(1,\alpha) |^{2} d\alpha \leq C k ^{-1}\Big(\parallel f\parallel ^2  _{(1)} (-1,1) \Big) 
 \end{equation}
\end{lemma}

\begin{proof}
Using \eqref{alphau1u-1}, we obtain
\begin{equation}
\label{Line1lemma4.1}
   \int _{k}^{\infty} \alpha ^2 | u(-1,\alpha) |^{2}d\alpha + \int _{k}^{\infty} \alpha ^2 | u(1,\alpha) |^{2}d\alpha 
\end{equation}

\begin{equation}
\label{Line2lemma4.1}
    \leq C \big ( \int _{k}^{\infty}\Big | \int _{0}^{1} e^{i \alpha y}  f_{1} (y)dy  \Big|^2d\alpha  + \int _{k}^{\infty}\Big | \int _{-1}^{0} e^{i \alpha y}  f_{2} (y)dy  \Big|^2d\alpha \Big ).
\end{equation}
By integration by parts and our assumption  $ supp f_1 \subset (0,1)  $ and $supp f_2 \subset (0,1) $, we derive
\begin{equation*}
\int _{0}^{1} e^{ -i \alpha y}  f _1(y)dy = \frac{1}{ i \alpha} \int _{0}^{1}e^{ -i\alpha y} ( \partial _ yf_{1}(y) )dy,  
\end{equation*}
and
\begin{equation*}
\int _{-1}^{0} e^{-i \alpha y}  f_{2} (y)dy = \frac{1}{i \alpha} \int _{-1}^{0}e^{ -i \alpha y}(\partial _ yf_{2} (y)  )dy, 
\end{equation*}
consequently for the first and second terms in \eqref{Line2lemma4.1} ,we have
\begin{equation*}
   \Big | \int _{0}^{1} e^{i \alpha y}  f_{1} (y)dy  \Big|^2 \leq  \frac{C}{\alpha^2}\parallel f_{1}\parallel^2 _{(1)} (0,1) \leq  \frac{C}{\alpha^2} \parallel f_{1}\parallel^2 _{(1)} (-1,1)
\end{equation*}
\begin{equation*}
    \leq  \frac{C}{\alpha^2} \parallel f\parallel^2 _{(1)} (-1,1),
\end{equation*}
utilizing the similar technique for the second term in \eqref{Line2lemma4.1} and integrating with respect to $\alpha$ the proof is complete.
\end{proof}
Finally, we are ready for the proof of the main Theorem \ref{maintheorem}.\\
\begin{proof}
without loss of generality, we can assume that $\epsilon <1$ and $3\pi E^{-\frac{1}{4}} <1$, otherwise the bound \eqref{maintheorem} is obvious. Let 's
\begin{center}
\begin{equation}
\label{k}
k= \begin{cases}
          K^{\frac{2}{3}}E^{\frac{1}{4}} \quad \text{if} \quad 2^{\frac{1}{4}}K^{\frac{1}{3}}< E ^{\frac{1}{4}} \\
K \hspace{1.19 cm} \text{if}\quad E ^{\frac{1}{4}} \leq 2^{\frac{1}{4}}K^{\frac{1}{3}},
       \end{cases} 
\end{equation}
 \end{center}
if $ E ^{\frac{1}{4}} \leq 2^{\frac{1}{4}}K^{\frac{1}{3}}$, then $k=K$, using the \eqref{I1Epsilon} and \eqref{I2Epsilon},  we can conclude 
\begin{equation}
\label{I11}
|I (k)| \leq 2\epsilon ^2.
 \end{equation}

If $2^{\frac{1}{4}}K^{\frac{1}{3}}< E ^{\frac{1}{4}}$, 
we can assume that $ E^{-\frac{1}{4}} <\frac{1}{4 \pi}$, otherwise $C<E$ and hence $K<C$ and the bound \eqref{Istability} is straightforward. From \eqref{k},  Lemma 2.2, \eqref{I1Epsilon} and the equality $\epsilon= \frac{1}{e^ E}$ we obtain
\begin{equation*}
\label{I1epsilon}
|I(k) |\leq  CM^{2} e^{4k}  e^{\frac{-2E}{\pi}\big( (\frac{k}{K})^4  -1 \big)^{\frac{-1}{2}}} 
\end{equation*}

\begin{equation*}
\leq CM^{2}  e^{- \frac{2}{\pi}K^{\frac{2}{3}}E^{\frac{1}{2}}(1- \frac{5\pi}{2}   E^{\frac{-1}{4}} )},
\end{equation*}
using the trivial inequality $e^{-t} \leq \frac{6}{t^3}$ for $t>0$ and our assumption at the beginning of the proof, we obtain 
\begin{equation}
\label{I12}
|I(k)|\leq  CM ^{2}  \frac{1}{K^2 E ^{\frac{3}{2}}\Big (1-\frac{5 \pi}{2} E^{-\frac{1}{4}}   \Big)^3}.  
\end{equation}
Due to the \eqref{I1}, \eqref{I11}, \eqref{I12}, and Lemma 2.5. we can conclude
\begin{equation}
\label{lastbound1}
\int ^{+\infty}_{0} \alpha ^2 |u(-1,\alpha)|^2 d\alpha + \int ^{+\infty}_{0} \alpha ^2 |u(1,\alpha)|^2 d\alpha
\end{equation}
\begin{equation*}
    \leq I(k)+  \int_{k}^{\infty} \alpha ^2 |u(-1,\alpha)|^2 d\alpha + \int_{k}^{\infty} \alpha ^2 |u(1,\alpha)|^2 d\alpha
\end{equation*}
\begin{equation*}
\leq 2\epsilon ^2 +   \frac{ C M^2 } {K^2 E ^{\frac{3}{2}}} + \frac{ \parallel f \parallel_{(2)}^2 (-1,1)} {K^{\frac{2}{3}}E^{\frac{1}{4}}+1} \Big).
\end{equation*}
Using the inequalities in \eqref{lastbound1} and Lemma 2.3., we finally obtain
\begin{equation*}
\parallel f \parallel _{(0)} ^2 (\alpha)\leq C \Big( \epsilon ^2 +   \frac{   M^{2} }{K^2 E ^{\frac{3}{2}}} + \frac{ \parallel f \parallel_{(1)}^2(-1,1) }{K^{\frac{2}{3}}E^{\frac{1}{4}}+1} \Big)
\end{equation*}
Due to the fact that $ K^{\frac{2}{3}} E ^{\frac{1}{4}}<K^2 E ^{\frac{3}{2}}$ for $1<K, 1<E$, the proof is complete.

\end{proof}

\section{Conclusion} In this work, we investigate the inverse source problem with many frequencies in a one dimensional domain with many frequencies. The result showed that if the wave number $K$ grows the estimate improves and the problems more stable. It also showed that if we have date exists for all wave number $k\in (0,\infty) $, the estimate will be a Lipschitz estimate.  It will be very interesting if we can investigate these type of problems in time domain with  different speed of propagation or in a domain with different number of layers and densities . It is also very amazing if one can study these type of problem in domain with corners. From computational point of view, these type of problems have a vast application in different area of science.  \\

\end{document}